\documentclass{amsart}
\usepackage{amssymb,latexsym}
\theoremstyle{plain}
\newtheorem{theorem}{Theorem}

\newtheorem{proposition}{Proposition}
\newtheorem{lemma}{Lemma}
\theoremstyle{definition}

\newtheorem{example}{Example}

\newtheorem{remark}{Remark}

\DeclareMathOperator{\red}{red}
\DeclareMathOperator{\reg}{reg}
\DeclareMathOperator{\Res}{Res}

\newcommand{\enm}[1]{\ensuremath{#1}}          %

\newcommand{\cal}[1]{\mathcal{#1}}

\newcommand{\ZZ}{\enm{\mathbb{Z}}}

\newcommand{\PP}{\enm{\mathbb{P}}}

\newcommand{\Ii}{\enm{\cal{I}}}

\newcommand{\Ll}{\enm{\cal{L}}}

\newcommand{\Oo}{\enm{\cal{O}}}

\newcommand{\Rr}{\enm{\cal{R}}}

\renewcommand{\phi}{\varphi}
\renewcommand{\theta}{\vartheta}
\renewcommand{\epsilon}{\varepsilon}

\begin{document}

\title[secant variety]
{Partially symmetric tensors and the non-defectivity of secant varieties of products with a projective line as a factor}
\author{Edoardo Ballico}
\address{Dept. of Mathematics\\
 University of Trento\\
38123 Povo (TN), Italy}
\email{edoardo.ballico@unitn.it}
\thanks{The author is a member of GNSAGA of INdAM (Italy).}

\begin{abstract}
We prove (with a mild restriction on the multidegrees) that all secant varieties of Segre-Veronese varieties with $k>2$ factors, $k-2$ of them being $\PP^1$, have the expected dimension.
This is equivalent to compute the dimension of the set of all partially symmetric tensors with a fixed rank and the same format. The proof uses the case $k=2$ proved by Galuppi and Oneto. Our theorem is an easy consequence of a theorem proved here for arbitrary projective varieties with a projective line as a factor and with respect to complete linear systems. 
\end{abstract}

\maketitle
\section{Introduction}
Let $Y$ be an integral projective variety over an algebraically closed field with characteristic $0$ and $\Ll$ a very ample line bundle of $Y$ such that $h^1(\Ll)=0$. Set $\alpha:= h^0(\Ll)$. For any positive integer $z$ and for any integral and non-degenerate variety 
$W\subset \PP^r$ let $\sigma _z(W)$ denote the $z$-secant variety of $W$, i.e. the closure of the union of all linear spaces spanned by $z$ points of $W$. Many practical linear algebra problems and applications use secant varieties or, at least, their dimensions (\cite{l}). For instance, if $\sigma _{z-1}(W)\ne \PP^r$, then the integer $\dim \sigma _z(W)$ is the dimension of the set of all $q\in \PP^r$ with $W$-rank $z$. If we take as $W$ a multiprojective space, then the $W$-rank decompositions are the partially symmetric tensor decompositions with the minimal number of addenda. If we take $W=\PP^n$, then the $W$-ranks correspond to the additive decompositions of forms in $n+1$ variables with a minimal number of addenda.

We see $Y\subset \PP^{\alpha -1}$ as embedded variety by the complete linear system $|\Ll|$.
Our assumptions are satisfied if $Y$ is not secant defective, i.e. if $\dim \sigma_z(Y) =\max\{\alpha -1,z(\dim Y+1)\}$ for all positive integer $z$. In our main result (Theorem \ref{i1}) we only require that $\sigma _z(Y)$ has dimension $(\dim Y+1)z-1$ for $z=\lfloor \alpha/(\dim Y+1)\rfloor$, i.e. for the largest integer $z$ such that $z(\dim Y +1)\le \alpha$. But the thesis is that all the secant varieties of the embedding of $X$ described in Theorems \ref{minus} and \ref{i1} have the expected dimension. 

Segre-Veronese varieties are related to partially symmetric tensors and hence their secant varieties (or at least their dimensions) are an active topic of research with several papers devoted to the study of the dimensions of their secant varieties (\cite{AB09a,AB09b,cgg1,go,lmr,lp}).
As a corollary of our results we prove the following result in which we only use the secant non-defectivity of almost all  Segre-Veronese varieties with $2$-factors (\cite{go}).

\begin{theorem}\label{minus}
Take $X:= \PP^{n_1}\times \PP^{n_2}\times (\PP^1)^{k-2}$, $k\ge 3$, embedded by the complete linear system $\Oo_X(d_1,\dots ,d_k)$ with $d_i\ge 2$ for all $i$, $d_1\ge 3$ and $d_2\ge 3$. Then this embedding of $X$ is not secant defective.
\end{theorem}

We prove a more interesting result, which applies to all varieties with a projective line as a factor and with respect to complete linear systems (Theorems \ref{i1}). 

Let $Y$ be an integral projective variety.
Set $X:= Y\times \PP^1$.  Let $\pi _1: X\to Y$ and $\pi_2: X\to \PP^1$ denote the $2$ projections. For any $\Ll\in \mathrm{Pic}(Y)$ and any $\Rr\in \mathrm{Pic}(\PP^1)$ set $\Ll\boxtimes \Rr:= \pi _1^\ast(\Ll)\otimes \pi_2^\ast(\Rr)$. Obviously, $\mathrm{Pic}(\PP^1) \cong \ZZ\Oo_{\PP^1}(1)$. Set $\Ll[t]:= \Ll\boxtimes \Oo_{\PP^1}(t)$.  Every line bundle on $X$ is of the form $\Ll[t]$ for a  uniquely determined
$\Ll\in \mathrm{Pic}(Y)$ and a unique $t\in \ZZ$ (\cite[Proposition 3]{fu}). Now assume $t\ge 0$ and $\alpha:= h^0(\Ll) >0$. The K\"{u}nneth formula gives $h^0(\Ll[t]) =(t+1)\alpha$ and $h^1(\Ll[t]) = (t+1)h^1(\Ll)$. From now on we also assume $h^1(\Ll)=0$.

\begin{theorem}\label{i1}
Let $Y$ be an integral projective variety.  Fix an integer $t\ge 2$.
Set $X:= Y\times \PP^1$. Let $\Ll$ be a very ample line bundle on $Y$ with $h^1(\Ll)=0$. Set $n:= \dim X$, $\alpha := h^0(\Ll)$ and $e_1:= \lfloor \alpha/n\rfloor$. Assume $n\ge 3$, $\alpha > n^2$ and  that the $e_1$-the secant variety
of $(Y,\Ll)$ has the expected dimension.  Then the pair $(X,\Ll[t])$ is not secant defective.
\end{theorem}

We assume $n\ge 3$  in Theorem \ref{i1} because if $n=t=2$ no lower bound on $\alpha$  may work, as shown by Example \ref{ex1}.

Theorem \ref{minus} is an easy consequence of Theorem \ref{i1}.

Our tools work even if $(Y,\Ll)$ is secant defective, adding conditions on $t$ and/or $\alpha$. As an example we prove one case in which we only assume that $\sigma_{\lfloor \alpha/n\rfloor-1}(Y)$ has the expected dimension.
(Theorem \ref{i1.0}). We also see that even for $t=1$ we may get non-trivial results (Proposition \ref{u1}). 

We often use the Differential Horace Lemma (\cite{ah1,ah}) and an inductive procedure (the Horace Method) but from top to bottom with smaller and smaller zero-dimensional schemes to be handled. This approach may be considered as a controlled asymptotic tool which do not require the low cases to start the inductive procedure (\cite{bb,bbcs}, \cite[Lemma 3]{c}). A long and detailed explanation
of this method is contained in \cite[pp. 1005-1008]{bb}; see in particular the diagram of logical implications in \cite[p. 1058]{bb}. Then sometimes the low cases may be proved, e.g. with a computer assisted proof (\cite{bbcs,d}). However, a standard use of this tool would only give a very weak result (e.g. Theorem \ref{i1} only for $t\ge \dim Y+2$ and hence the inductive proof to get Theorem \ref{minus} for $k\ge 4$ would require very large $t$). 

For our proof of Theorem \ref{i1} the key part is the proof of the case $t=2$. Then the cases $t>2$ have a short inductive proof using Lemmas \ref{a2}, \ref{a3} and \ref{a4.0}.

Our method works only for some non-complete linear systems (see Remark \ref{fin1}).

The author has no conflict of interests.
\section{Preliminaries}
Let $W$ a projective variety and $D$ an effective Cartier divisor of $W$.
For any $p\in W_{\reg}$ let $(2p,W)$ denote the closed subscheme of $W$ with $(\Ii_{p,W})^2$ as its ideal sheaf. We have $\deg ((2p,W)) =\dim W+1$ and $(2p,W)_{\red} =\{p\}$. For any finite set $S\subset W_{\reg}$ set $(2S,W):= \cup_{p\in S} (2p,W)$.
We often write $2p$ and $2S$ instead of $(2p,X)$ and $(2S,X)$. For any zero-dimensional scheme $Z\subset W$ let $\Res_D(Z)$ denote the residual scheme of $Z$ with respect to $D$, i.e. the closed subscheme of $W$ with $\Ii_Z:\Ii_D$ has its ideal sheaf.
We have $\deg (Z) =\deg (\Res_D(Z)) +\deg (Z\cap D)$, $\Res_D(Z)\subseteq Z$, $\Res_D(Z) =Z$ if $Z\cap D=\emptyset$, $\Res_D(Z) =\emptyset$ if $Z\subset D$ and $\Res_D(Z) =\Res_D(A)\cup \Res_D(B)$ if $Z=A\cup B$ and $A\cap B=\emptyset$. If $p\in D_{\reg}\cap W_{\reg}$, then $\Res_D((2p,W)) = \{p\}$ and $(2p,W)\cap D =(2p,D)$. For any line bundle $\Rr$ on $W$ there is an exact sequence
\begin{equation}\label{eqp1}
0 \to \Ii_{\Res_D(Z)}\otimes \Rr(-D)\to \Ii_Z\otimes \Rr \to \Ii_{Z\cap D,D}\otimes \Rr_{|D} \to 0
\end{equation}
of coherent sheaves on $W$ which we call the {\it residual sequence of $D$}. Fix a positive integer $z$. By the Terracini Lemma (\cite[Cor. 1.11]{a}, \cite[5.3.1.1]{l}) the integer $\dim \sigma _z(W)$ is the codimension of the linear span of the zero-dimensional scheme $(2S,W)$, where $S$ is a general subset of $W$ of cardinality $z$. Thus if the embedding of $W$ is induced by the complete linear system $|\Rr|$, then $\dim \sigma _z(W) =h^0(\Rr)-1-h^0(\Ii _{(2S,W)}\otimes \Rr)$.
Hence $\dim \sigma _z(W) = z(\dim W+1)-1$ if and only if $h^1(\Ii _{(2S,W)}\otimes \Rr) =h^1(\Rr)$.

Now we describe the so-called  Differential Horace Lemma (\cite{ah1,ah}).

\begin{remark}\label{dh1}
Let $E\subset W$ be a zero-dimensional scheme. Fix $i\in \{0,1\}$ and an integer $g>0$. Let  $D$ be an integral divisor of $W$. Let $F\subset D$ be a general subset of $D$ with $\#F=g$. Suppose you want to prove that $h^i(\Ii_{E\cup (2S,W)}\otimes \Rr)=0$ for a general $S\subset W$ such that $\#S =g$. 
It is sufficient to prove that $h^i(H,\Ii _{(E\cap D)\cup F}\otimes \Rr_{|D})=0$ and $h^i(W,\Ii_{\Res_D(E)\cup (2F,D)}\otimes \Rr(-D))=0$ (\cite{ah1,ah}).
\end{remark}

Suppose there is a line bundle $\Rr$ on $W$ such that the embedding $W\subset \PP^r$ is induced by the complete linear system $|\Rr|$. We call the secant varieties of the embedded variety $W\subset \PP^r$ the secant varieties of the pair $(W,\Rr)$.

\begin{remark}\label{ovv1}
Let $W\subset \PP^r$ be an integral and non-degenerate variety. Set $n:= \dim W$, $z_1:= \lfloor (r+1)/(n+1)/\rfloor$ and $z_2:=\lceil (r+1)/(n+1)\rceil$. Suppose that $\sigma _{z_1}(W)$ and $\sigma _{z_2}(W)$ have the expected dimension. Since $(n+1)z_2\ge r+1$, $\sigma _x(W) =\PP^r$ for all $x\ge z_2$. Either $z_1=z_2$ or $z_1=z_2-1$. Let $S\subset W_{\reg}$ be a general subset such that $\#S=z_1$. Since $\dim \sigma _{z_1}(W) =(n+1)z_1-1$, the Terracini Lemma
gives $\dim (2S,W) =(n+1)z_1-1$, i.e. the scheme $(2S,W)$ is linearly independent.
Hence $(2A,W)$ is linearly independent for any $A\subset S$. The Terracini Lemma gives $\dim \sigma _y(W) =(n+1)y-1$ for all $1\le y<z_1$ (a similar statement is proved in \cite[Prop. 2.1(i)]{a1}). Thus $W$ is not secant defective. Note that $z_2$
is the minimal integer $z$ such that $(\dim W+1)z \ge r+1$. Thus to prove that $W$ is not secant defective it is sufficient to test all positive integers $z$ such that $(\dim W+1)z\le r+1+\dim W$.
\end{remark}

\section{The proofs and related results}
Set  $X:= Y\times \PP^1$ and $n:= \dim X$. We fix $o\in \PP^1$ and set $H:= Y\times \{o\}$. The set $H$ is an effective Cartier divisor of $X$ and $H\cong Y$.

For all positive integers $t$ and $z$ call $A(t,z)$ the following statement:

\quad $A(t,z)$: We have $h^0(\Ii _{2S}\otimes \Ll[t]) =\max \{\alpha(t+1)-(n+1)z,0\}$ for a general $S\subset X$ such that $\#S =z$.

By the Terracini Lemma $A(t,z)$ is true if and only if the $z$-secant variety of the pair $(X,\Ll[t])$ has the expected dimension.
Since $h^1(\Ll[t])=0$, $A(t,z)$ is equivalent to $h^1(\Ii _{2S}\otimes \Ll[t]) =\max \{0, (n+1)z -\alpha(t+1)\}$ for a general $S\subset X$ such that $\#S =z$.

We say that $A(t)$ is true if $A(t,z)$ are true for all $z\in \{\lfloor (t+1)\alpha/(n+1)\rfloor,\lceil (t+1)\alpha/(n+1)\rceil\}$. Remark \ref{ovv1} shows that $A(t)$ is true if and only if $A(t,z)$ is true for all positive integers $z$.

Write $\alpha = ne_1+f_1$ with $e_1, f_1$ integers and $0\le f_1\le n$, i.e. set $e_1:= \lfloor \alpha/n\rfloor$ and $f_1:= \alpha -ne_1$.

The following example is well-known (\cite[p. 1457]{lp}).
\begin{example}\label{ex1}
Fix a positive integer $a$, $t=2$ and $(Y,\Ll)=(\PP^1,\Oo _{\PP^1}(2a))$. We have $n=2$, $\alpha =2a+1$, $(X,\Ll[2])$ is secant defective with only defective the $(2a+1)$-secant variety. Indeed, a general $S\subset \PP^1\times \PP^1$ with $\#S =2a+1$
is contained in the singular locus of a unique $D\in |\Oo_{\PP^1\times \PP^1}(2a,2)|$, the double curve $2T$ with $T$ the unique element of $|\Ii_S(a,1)|$.
\end{example}

\begin{remark}\label{aa2}
It is very important for our proof  that $f_1\le e_1$. Since $f_1\le n-1$, it is sufficient to assume $e_1\ge n-1$. If $\dim \sigma_{e_1}(Y) =ne_1-1$, it is sufficient to assume $\alpha \ge n(n-1)$, which is quite mild.
\end{remark}

For all positive integers $t$ and $z$ we call $B(t,z)$ the following statement:

\quad $B(t,z)$: Either $z< e_1+f_1$ or $h^0(\Ii_E\otimes \Ll[t-1]) =\max \{0,t\alpha -(n+1)(z-e_1-f_1)-ne_1-f_1\}$, where $E\subset X$ is a general union of $z-e_1-f_1$ double points of $X$, $f_1$ double points of $H$ and $e_1$ points of $H$.

For all $t\ge 2$ we call $C(t,z)$ the following statement:

\quad $C(t,z)$: We have $h^0(\Ii_{W}\otimes \Ll[t-2])\le \max \{0,(t-1)\alpha -\deg (W)\}$, where $W$ is a general union of $\max \{0,z-e_1-f_1\}$ double points of $X$.

Note that $C(2,z)$ is true if and only if $z\le e_1+f_1$.

We say that $B(t)$ (resp. $C(t)$) is true if $B(t,z)$ (resp. $C(t,z)$) is true for all $z$.

Using the Differential Horace Lemma it is quite easy to get $A(t)$ if we know $B(t)$. A key step is to get $B(t)$ knowing that $C(t)$ is true. Indeed, for $z\le \lceil h^0(\Ll[t])/(n+1)\rceil$ the integer $z-e_1-f_1$ usually degree much smaller
than $\lfloor h^0(\Ll[t-2])/(n+1)\rfloor$ and hence it should be `` easy '' to prove that $h^1(\Ii_W\otimes \Ll[t-2])=0$. But of course, $t-2<t$ and so to use this strategy we need to prove that $\lfloor h^0(\Ll[t-2])/(n+1)\rfloor - (z-e_1-f_1)$ is very large (depending on $t$). For $t=2$ we also need another trick. Then we prove $C(3)$. Then Lemmas \ref{a2}, \ref{a3}, \ref{a4.0} give the case $t\ge 4$ by induction on $t$.

\begin{lemma}\label{a2}
Assume that the $e_1$-secant variety of $(Y,\Ll)$ has dimension $ne_1-1$. If $t\ge 2$ and $B(t,z)$ is true, then $A(t,z)$ is true. 
\end{lemma}

\begin{proof}
Let $Z\subset X$ be a general union of $z$ double points. 

First assume $z\ge e_1+f_1$. Let $Z'\subset X$ be a general union of $z-e_1-f_1$ double points. Fix a general $S\subset H$ such that $\#S =e_1+f_1$ and write $S=S'\cup S''$ with $\#S' =e_1$ and $\#S''=2$.
Set $A:= (2S',H)\cup S''$ and $B:= S'\cup (2S'',H)$. Note that $\deg (A)=\alpha$. Since $\dim \sigma _{e_1}(Y) =e_1n-1$, $h^1(H,\Ii _{(2S',H)}\otimes \Ll[t]_{|H}) =0$. Since $S''$ is general in $H\cong Y$
and $\deg (A)=\alpha$, 
$h^i(H,\Ii _{A,H}\otimes \Ll[t]_{|H}) =0$, $i=0,1$. The Differential Horace Lemma (Remark \ref{dh1}) gives $h^i(\Ii _Z\otimes \Ll[t])) =h^i(\Ii_{Z'\cup S'\cup (2S'',H)}\otimes \Ll[t-1])$ for $i=0,1$. Thus $B(t,z)$ implies $A(t,z)$ if $z\ge e_1+f_1$.

Now assume  $z\le e_1+f_1-1$. The proof of the case $z\ge e_1+f_1$ works taking $Z'= \emptyset$, $\#S' =\min \{e_1,z\}$ and $\#S'' = z-\#S'$.
\end{proof}
\begin{lemma}\label{a3}
Assume $\alpha >n^2$, $t\ge 3$ and that the $e_1$-secant variety of $(Y,\Ll)$ has dimension $ne_1-1$. Take $z\le \lceil (t+1)\alpha /(n+1)\rceil$. If $C(t,z)$ is true and either $z\le e_1+f_1$ or $A(t-2,z-e_1-f_1)$ is true, then $B(t,z)$ and $A(t,z)$ are true.
\end{lemma}

\begin{proof}
By Remark \ref{ovv1} it is sufficient to check all positive integers $z$ such that $(n+1)z \le (t+1)\alpha +n$.
By Lemma \ref{a2} it is sufficient to prove $B(t,z)$. By the definition of $B(t,z)$ we may assume $z\ge e_1+f_1$. Let $W\subset X$ be a general union of $z-e_1-f_1$ double points. Take a general $S\subset H$ such that
$\#S =e_1+f_1$ and write $S =S'\cup S''$ with $\#S'= e_1$ and $\#S''=f_1$. By assumption $h^0(\Ii_W\otimes \Ll[t-2])\le \max \{0,(t-1)\alpha -\deg(W)\}$,
i.e. either $h^0(\Ii _W\otimes \Ll[t-2])=0$ or $h^1(\Ii _W\otimes \Ll[t-2]) =0$. Assume that $B(t,z)$ fails. Hence $h^1(\Ii _{W\cup (2S'',H)\cup S'}\otimes \Ll[t-1]) >0$. 

\quad (a) Assume $h^1(\Ii _W\otimes \Ll[t-2])) =0$. Since $f_1\le e_1$, Remark \ref{ovv1} gives $\dim \sigma _{f_1}(Y)) =nf_1-1$. Thus $h^1(H,\Ii _{(2S'',H),H}\otimes \Ll[t-1]_{|H})=0$. Thus the residual exact sequence of $H$
gives $h^1(\Ii _{W\cup (2S'',H)}\otimes \Ll[t-1]) =0$. Let $a$ be the maximal integer such that $h^1(\Ii _{W\cup (2S'',H)\cup A}\otimes \Ll[t-1]) =0$, where $A$ is a general subset of $H$ with cardinality $a$. By assumption $a<e_1$.
Take a general $p\in H$. The definition of $a$ gives $h^1(\Ii _{W\cup (2S'',H)\cup A\cup \{p\}}\otimes \Ll[t-1]) >0$. Since $h^1(\Ii _{W\cup (2S'',H)\cup A}\otimes \Ll[t-1]) =0$, we see that $H$ is in the base locus of $|\Ii _{W\cup (2S'',H)\cup A}\otimes \Ll[t-1]|$,
i.e, (since $\Res_H(W\cup (2S'',H)\cup A) =W$ and $h^1(\Ii _W\otimes \Ll[t-2]) =0$), $h^0(\Ii _{W\cup (2S'',H)\cup A}\otimes \Ll[t-1]) =(t-1)\alpha -\deg (W)$. We get $a+nf_1=\alpha$, which is false because $a<e_1$, $e_1\ge f_1$ (Remark \ref{aa2})
and $\alpha =ne_1+f_1$.

\quad (b) Assume $h^0(\Ii _W\otimes \Ll[t-2]) =0$. We get $(n+1)(z-e_1-f_1)\ge (t-1)\alpha$. By assumption $(n+1)z \le (t+1)\alpha +n$. Hence $2\alpha \le (n+1)e_1+(n+1)f_1+n$. We have $ne_1+f_1 =\alpha$. Thus $\alpha\le nf_1+n\le n^2$, a contradiction.\end{proof}

\begin{lemma}\label{a4.0}
Fix an integer $t\ge 3$ and assume $A(t-2)$. Then $C(t)$ is true.
\end{lemma}

\begin{proof}
Fix a positive integer $z$ for which we want to prove $C(t,z)$. If $z<e_1+f_1$, then $C(t,z)$ is true. Now assume $z\ge e_1+f_1$ and let $W\subset X$ be a general union of $z-e_1-f_1$ double points of $X$. If $\deg (W)\ge h^0(\Ll[t-2])$, we get
$h^0(\Ii _W\otimes \Ll[t-2])=0$ by $A(t-2)$. If $\deg (W) <h^0(\Ll[t-2])$, then $h^1(\Ii_W\otimes \Ll[t-2])=0$, i.e. $h^0(\Ii_W\otimes \Ll[t-2])=(t-1)\alpha -\deg(W)$. Thus $C(t,z)$ is true.\end{proof}

\begin{proposition}\label{u1}
Fix a positive integer $z$ such that $n(z-e_1)+e_1\le \alpha$ and assume that the $e_1$-secant variety of $(Y,\Ll)$ has dimension $ne_1-1$. Then the $z$-secant variety of the pair $(X,\Ll[1])$ has dimension $z(n+1)-1$.
\end{proposition}

\begin{proof}
It is sufficient to do the cases with $z\ge e_1$. Take a general $(A,B)\subset H\times H$ such that $\#A=z-e_1$ and $\#B =e_1$. By assumption $\#A+\#B \le \alpha$. Since the $e_1$-secant variety of $(Y,\Ll)$ has dimension $ne_1-1$ and $e_1$ is a non-negative integer,
$h^1(H,\Ii_{(2A,H)\cup B,H}\otimes \Ll)=0$. By the Differential Horace Lemma (Remark \ref{dh1}) to prove the proposition it is sufficient to prove that $h^1(\Ii_{(2B,H)}\otimes \Ll)=0$. Since $(2B,H) \subset H$ and $h^1(\Ll[-1]) =0$ by the K\"{u}nneth formula, we have $h^1(\Ii_{(2B,H)}\otimes \Ll)=h^1(H,\Ii_{(2B,H),H}\otimes \Ll_{|H})=0$ (because the $e_1$-secant variety of $(Y,\Ll)$ has dimension $ne_1-1$).\end{proof}

\begin{proof}[Proof of Theorem \ref{i1}:]
By Remark \ref{ovv1} it 
is sufficient to test all positive integers $z$ such that  $(n+1)z\le (t+1)\alpha +n$.

\quad {\bf Outline of the proof:} We start with the proof of some numerical inequalities. Recall that $B(t,z)$ implies $A(t,z)$ for $t\ge 3$ (Lemma \ref{a3}). We first prove the theorem for $t=2$  without using Lemma \ref{a3} (step (a)).
Then we take $t=3$ and prove $C(3)$ and $B(3)$ (step (b)). Then for $t\ge 4$ we prove the theorem by induction on $t$  using that $A(t-2)$ is proved.

Let $S\subset H$ (resp. $S'\subset H$, resp. $S''$) be a general subset of $H$ with cardinality $e_1$ (resp. $f_1$, resp. $e_1-f_1$). Let $Z\subset X$ be a general union of $z$ double points of $X$. By the Terracini Lemma it is sufficient to prove that
either $h^0(\Ii_Z\otimes \Ll[t])=0$ or $h^1(\Ii_Z\otimes \Ll[t])=0$. Since $f_1\le n-1$, $f_1+ne_1 = \alpha$ and $\alpha \ge n^2-1$, we have $f_1\le e_1$. Thus the $f_1$-secant variety of the pair $(Y,\Ll)$ has dimension $f_1n-1$ (Remark \ref{ovv1}).
Note that $h^i(Y,\Ii _{(2A,Y)\cup B}\otimes \Ll)=0$, $i=0,1$, where $A$ is a general subset of $Y$ with cardinality $e_1$ and $B$ is a general subset of $Y$ with cardinality $f_1$.
 Set $\Delta := (t+1)\alpha +n -z(n+1)$ and $w:= z-(t-1)e_1-f_1$.

\quad {\bf Claim 1:} We have 
\begin{equation}\label{eqnn5}
\Delta + z \ge tf_1+e_1+n
\end{equation}

\quad {\bf Proof of Claim 1:} We have $\Delta +(n+1)z = (t+1)\alpha +n$. Thus $(n+1)\Delta +(n+1)z = n\Delta +(t+1)\alpha +n$. Since $\alpha =ne_1+f_1$, to prove Claim 1
 it is sufficient to prove that $(t+1)ne_1+(t+1)f_1 - n\ge (n+1)tf_1+(n+1)e_1+n(n+1)$,
 i.e. $f_n(t) \ge 0$, where $f_n(t):= (t+1)ne_1 +(t+1)f_1-(n+1)tf_1-(n+1)e_1-n(n+2)$.  We have $f_n(2) =(2n-1)ne_1+(1-2n)f_1-n(n+2)$. Since $f_1\le e_1$, $f_n(2) \ge (2n-1)(n-1)e_1-n(n+2)$. Since $\alpha \ge n^2$, $e_1\ge n$.
 Since $n\ge 3$, $f_n(2)\ge 0$.
 The derivative $f'_n(t)$ of $f_n(t)$ is $ne_1-(n+1)f_1$. Since $f_1\le n-1$, $f'_n(t)\ge 0$ if $e_1\ge n$, i.e. if $\alpha \ge n^2$.

\quad {\bf Claim 2:} We have $nf_1+n(w-e_1)+e_1\le \alpha$.

\quad {\bf Proof of Claim 2:} By Remark \ref{aa2} we may assume $w\ge e_1$, i.e. $z\ge te_1+f_1$. Since $w=z-(t-1)e_1-f_1$, $w-e_1=z-te_1-f_1$. Thus $nf_1+n(w-e_1) +e_1 = (n+1)z-z -t\alpha +tf_1 +1$. Since $(n+1)z =(t+1)\alpha +n-\Delta$, we have $nf_1+n(w-e_1) +e_1 \le \alpha$ by \eqref{eqnn5}.

\quad (a) Assume $t=2$. We assume $z\ge e_1+f_1$, since the case $z< e_1+f_1$ only requires a small modification and it is never a critical case $z\in \{\lfloor 3\alpha /(n+1)\rfloor, \lceil 3\alpha/(n+1)\rceil\}$.
Set $w:= z-e_1-f_1$. We write $E\cup Z'$ with $E\cap Z' =\emptyset$ , $E$ union of $w-e_1-f_1$ general double points and $Z'$ a general union of $e_1+f_1$ double points.
We use Differential Horace  Lemma (Remark \ref{dh1}) $e_1+f_1$ times to the connected  of $Z'$ with respect to the general subset $S'$ of $H$. To prove that $h^0(\Ii_Z\otimes \Ll[2]) =\max \{0,3\alpha -(n+1)z\}$ it is sufficient to prove that $h^0(\Ii_{E\cup S\cup (2S',H)}\otimes \Ll[1]) =\max \{0,2\alpha -(n+1)(z-e_1-f_1) - nf_1 -e_1)\}$.
Since $e_1\ge n \ge (n+1)z-3\alpha$ and $S$ is general in $H$, it is sufficient to prove that $h^1(\Ii_{E\cup (2S',H)}\otimes \Ll[1]) =0$ and that $h^0(\Ii_E\otimes \Ll[0]) \le \max \{0,h^0(\Ii_{E\cup (2S',H)}\otimes \Ll[1]) -e_1\}$. The last inequality is critical for the proof (see Claim 4) and it fails
in Example \ref{ex1} with $n=2$, $z=2a+1$, $e_1=a$, $f_1=0$, because $h^0(\Ii_E\otimes \Ll[0]) =1$ and (assuming $h^1(\Ii_{E\cup (2S',H)}\otimes \Ll[1]) =0$) $h^0(\Ii_{E\cup (2S',H)}\otimes \Ll[1]) -e_1 =4a+2-3a-2-a=0$.

\quad {\bf Claim 3:} Either $h^1(\Ii_{E\cup (2S',H)}\otimes \Ll[1]) =0$ or  $h^0(\Ii_{E\cup (2S',H)}\otimes \Ll[1]) =0$.

\quad {\bf Proof of Claim 3:} Take a general $S_3\subset H$ such that $\#S_3 = w-e_1$ and a general $S_4\subset H\setminus S$ such that $\#S_4 =e_1$.  By the Differential Horace Lemma to prove Claim 3 it is sufficient to prove that  $h^1(\Ii _{2S_4,H}\otimes \Ll[0]) =0$ and either $h^1(H,\Ii _{(2S',H)\cup (2S_3,H)\cup S_4}\otimes \Ll[1]_{|H})=0$ or $h^0(H,\Ii _{(2S',H)\cup (2S_3,H)\cup S_4}\otimes \Ll[1]_{|H})=0$.
We have $h^1(\Ii _{2S_4,H}\otimes \Ll[0]) =h^1(H,\Ii_{(2S_4,H)}\otimes \Ll) =0$ (Remark \ref{aa2}). We have $h^1(H,\Ii _{(2S',H)\cup (2S_3,H)}\otimes \Ll[1]_{|H}) =0$, because $\#S'+\#S_3 = w-e_1+f_1\le e_1$
by Claim 2. Since $S_4$ is general in $H$, either $h^1(H,\Ii _{(2S',H)\cup (2S_3,H)\cup S_4}\otimes \Ll[1]_{|H})=0$ or $h^0(H,\Ii _{(2S',H)\cup (2S_3,H)\cup S_4}\otimes \Ll[1]_{|H})=0$.

If $h^0(\Ii_{E\cup (2S',H)}\otimes \Ll[1]) =0$, then $h^0(\Ii_Z\otimes \Ll[2]) =0$, proving the theorem in this case. Thus we may assume $h^1(\Ii_{E\cup (2S',H)}\otimes \Ll[1]) =0$, i.e. $h^0(\Ii_{E\cup (2S',H)}\otimes \Ll[1]) =2\alpha -\deg (E) -nf_1$.

\quad {\bf Claim 4:} $h^0(\Ii_E\otimes \Ll[0]) \le \max \{0,h^0(\Ii_{E\cup (2S',H)}\otimes \Ll[1]) -e_1\}$.

\quad {\bf Proof of Claim 4:} We saw that we may assume  $h^0(\Ii_{E\cup (2S',H)}\otimes \Ll[1]) =2\alpha -\deg (E) -nf_1$. Since $E$ is a general union of some, $\gamma$, double points of $X$, $h^0(\Ii_E\otimes \Ll[0]) = h^0(Y,\Ii _U\otimes \Ll)$, where $U$ is a general union of $\gamma$ double points of $Y$.
First assume $\gamma \le e_1$. The assumption on $(Y,\Ll)$ gives $h^0(Y,\Ii _U\otimes \Ll) =\alpha -n\gamma$ (Remark \ref{ovv1}). We have  $\deg (E) =(n+1)\gamma$. Thus in this case it is sufficient to check that $\alpha \ge \gamma +nf_1$. Since $\gamma \le e_1$,
it is sufficient to use that $ne_1+f_1\ge e_1+nf_1$ (Remark \ref{aa2}). Now assume $\gamma >e_1$. A general union $\Gamma$ of $e_1$ double points of $Y$ and $f_1$ points of $Y$ satisfies $h^i(Y,\Ii_\Gamma \otimes \Ll)=0$, $i=0,1$. Thus $h^0(Y,\Ii_U\otimes \Ll) =0$
if $\gamma \ge e_1+f_1$, while if $e_1<\gamma <\gamma +f_1$, then $h^0(Y,\Ii _U\otimes \Ll)\le f_1-(\gamma -e_1)$. If $e_1<\gamma <\gamma +f_1$ we have $h^0(\Ii_{E\cup (2S',H)}\otimes \Ll[1]) =2\alpha -(n+1)\gamma-nf_1 =\alpha -\gamma -(n-1)f_1 \ge \alpha -e_1-nf_1-1$. It is sufficient to have $\alpha \ge e_1+(n+1)f_1-2$. Since $\alpha =ne_1+f_1$, it is sufficient to have $(n-1)e_1+2 \ge nf_1$, i.e. $n(n-1)e_1+2n\ge n^2f_1$. Since $f_1\le (n-1)$, it is sufficient to have $ne_1 +2n/(n-1) \ge n^2$, which is true for $\alpha >n^2$.

Claims 3 and 4 prove the case $t=2$.

\quad (b) In this step we prove $C(3,z)$. We may assume $z\ge e_1+f_1$. Let $W\subset X$ be a general union of $w':= z-e_1-f_1$ double points.
We need to prove that $h^0(\Ii_W\otimes \Ll[1]) \le \max \{0,3\alpha -\deg(W)\}$. If $w'\ge 2e_1+f_1$, using twice the Differential Horace Lemma with respect to $H$ we get  $h^0(\Ii_W\otimes \Ll[1])) =0$.  If $w'\le e_1$ using once the usual Horace Lemma we get $h^1(\Ii_W\otimes \Ll[1])=0$. If $e_1\le w' \le 2e_1$, using twice the usual Horace Lemma we get $h^1(\Ii_W\otimes \Ll[1])\le 2e_1-w'$, which is enough to get $C(3,z)$. If $w' =2e_1+1$ we get $h^0(\Ii_W\otimes \Ll[1]) =0$ in the following way. We first degenerate $W$ to $W'\cup W''$ with $W'\cap W''=\emptyset$, $W'$ a general union of $e_1$ double points of $X$ and $W''$ a general union of $e_1+1$ double points of $X$ with $S_5:= W''_{\red}\subset H$. Since $(Y,\Ll)$ is not secant defective, $h^0(H,\Ii _{W''\cap H,H}\otimes \Ll[1]_{|H})=0$. Thus the residual exact sequence of $H$ gives
$h^0(\Ii_W\otimes \Ll[1])\le h^0(\Ii_{W'\cup S_5}\otimes \Ll[0])$. We have $h^0(\Ii_{W'\cup S_5}\otimes \Ll[0])=h^0(H,\Ii_{W'\cap H\cup S_5}\otimes \Ll[0]_{|H}) =0$, because $e_1\ge f_1$.

\quad (c) Step (b) and Lemma \ref{a3} proves $A(3)$. Thus  we may assume $t\ge 4$ and that $A(x)$ is true for $2\le x <t$. By Lemma \ref{a3} it is sufficient to prove $C(t)$. Since $A(t-2)$ is true, Lemma \ref{a4.0} gives $C(t)$.\end{proof}

Sometimes we get $(X,\Ll[t])$ not secant defective even if $\dim \sigma _{e_1}(Y) \le ne_1-2$, but we need to add other conditions. We give the following example in which we only assume that the $(e_1-1)$-secant variety of the pair
$(Y,\Ll)$ has dimension $n(e_1-1)-1$.

\begin{theorem}\label{i1.0}
Fix  integer $t\ge 2$ and $z>0$.  Assume $n\ge 3$, $\alpha \ge 2n^2+4n$, $t\ge 2$, $\dim \sigma _{e_1-1}(Y)=n(e_1-1)-1$. Then the $z$-secant variety of $(X,\Ll[t])$ has the expected dimension.
\end{theorem}

\begin{proof}
The proof is very similar to the one of Theorem \ref{i1}, just using the integers $g_1:= e_1-1$ and $h_1:= f_1+n$ instead of $e_1$ and $f_1$. The stronger assumption on $\alpha$ comes from the inequality $h_1\le 2n-1$ instead of the inequality
$f_1\le n-1$.
\end{proof}

\begin{proof}[Proof of Theorem \ref{minus}:]
We have $n:= \dim X =n_1+n_2+k$. 

First assume $k=3$.  We apply Theorem \ref{i1} to $Y:= \PP^{n_1}\times \PP^{n_2}$ with $\alpha =\binom{n_1+d_1}{n_1}\binom{n_2+d_2}{n_3}$. Since $\alpha$ is an increasing function of $d_1$ and $d_2$, it is sufficient to check that $\alpha \ge n^2$ if $d_1=d_2=3$. In in this case $\alpha -n^2-1 =\binom{n_1+3}{3}\binom{n_2+3}{3} -(n_1+n_2+1)-1$. Taking the derivatives with respect to $n_1$ and $n_2$ we see that it is sufficient to check the case $n_1=n_2=1$. In this case $n=3$ and $\alpha =16$.

Now assume $k>3$ and that $Y:= \PP^{n_1}\times \PP^{n_2}\times (\PP^1)^{k-1}$ is not secant defective. We apply Theorem \ref{i1} to $Y$ with $\alpha =\binom{n_1+d_1}{n_1}\binom{n_2+d_2}{n_3}(d_3+1)\cdots (d_{k-1}+1)\ge 3^{k-1}\binom{n_1+d_1}{n_1}\binom{n_2+d_2}{n_3}$.
We immediately get $\alpha > n^2$.
\end{proof}

\begin{remark}\label{fin1}
Fix a non-degenerate embedding $Y\subset \PP^r$ for which we know that many secant varieties have the expect dimension. Set $\Ll:= \Oo_Y(1)$ and $X:= Y\times \PP^1$. Let $V\subset H^0(\Oo_Y(1))$ denote the image of the restriction map $H^0(\Oo_{\PP^r}(1)) \to H^0(\Ll)$. For any integer $t>0$ let $V[t]$ denote the linear subspace $\pi_1^\ast(V)\otimes \pi_2^\ast(\Oo_{\PP^1}(t))$ of $\Ll[t]$. We have $\dim V[t] = (r+1)(t+1)$ and we may use the proofs of this paper with $r+1$ instead of $\alpha$. Hence, under the assumptions of one of the results we get the non-defectivity of the secant varieties with respect to the embedding induced by a general linear subspace  of dimension $(r+1)(t+1)$ of $H^0(\Ll[t])$.
\end{remark}

\providecommand{\bysame}{\leavevmode\hbox to3em{\hrulefill}\thinspace}

\end{document}